\documentclass[11pt, reqno]{amsart}
\usepackage{amsmath,amssymb,amsfonts,amscd,hyperref,color}
\usepackage[utf8]{inputenc}
\usepackage{csquotes}
\usepackage{nicefrac}
\usepackage[abs]{overpic}
\usepackage{verbatim}
\usepackage{enumerate}

\usepackage[
backend=biber,
style=alphabetic,
sorting=nyt,
maxnames = 99,
minnames = 99,
maxalphanames=99,
maxcitenames=99,
maxbibnames = 99,
doi=false,
url=false,
isbn= false
]{biblatex}

\AtEveryBibitem{%
	\clearfield{issn}%
}

\newcommand{\Hmm}[1]{\leavevmode{\marginpar{\tiny%
$\hbox to 0mm{\hspace*{-0.5mm}$\leftarrow$\hss}%
\vcenter{\vrule depth 0.1mm height 0.1mm width \the\marginparwidth}%
\hbox to
0mm{\hss$\rightarrow$\hspace*{-0.5mm}}$\\\relax\raggedright #1}}}

\newtheorem{theorem}{Theorem}
\newtheorem{corollary}[theorem]{Corollary}
\newtheorem{lemma}[theorem]{Lemma}
\newtheorem{proposition}[theorem]{Proposition}

\theoremstyle{definition}

\newtheorem*{remark}{Remark}

\numberwithin{equation}{section}
\newcommand{\Z}{{\mathbb Z}}
\newcommand{\R}{{\mathbb R}}

\newcommand{\N}{{\mathbb N}}

\newcommand{\abs}[1]{\left| #1 \right|}

\newcommand{\rw}{transition weight}

\newcommand{\set}[1]{\left\lbrace #1 \right\rbrace }

\begin{document}
\title[Optimal Hardy Inequalities for Random Walks on $\Z^2$]
{Optimal Hardy Inequalities for Random Walks on $\Z^2$}
\author[P.~Hake]{Philipp Hake}
\address{P.~Hake, Institut f\"ur Mathematik, Universit\"at Leipzig, 04109 Leipzig, Germany}\email{philipp.hake@math.uni-leipzig.de}
\author[M.~Keller]{Matthias Keller}
\address{M.~Keller, Israel Institute of Advanced Studies, Jerusalem, Israel; Institut f\"ur Mathematik, Universit\"at Potsdam
14476  Potsdam, Germany}
\email{matthias.keller@uni-potsdam.de}
\author[F.~Pogorzelski]{Felix Pogorzelski}
\address{ F.~Pogorzelski, Israel Institute of Advanced Studies, Jerusalem, Israel;  Institut f\"ur Mathematik, Universit\"at Leipzig, 04109 Leipzig, Germany}
\email{felix.pogorzelski@math.uni-leipzig.de}

\date{\today}

\begin{abstract}
We prove an optimal Hardy inequality for every aperiodic, symmetric random walk in $\Z^2$ with finite variance. In particular, we verify null-criticality, and thus, optimality of the underlying Hardy weight.
Under suitable moment conditions, we use fine asymptotics of the potential kernel due to Fukai and Uchiyama in order to derive the asymptotics of the weight.
For the standard Laplacian, we recover the expected first order term in the asymptotics but also show that next order term is negative. Thus, our result shows that the constant in the Hardy inequality proven by Kapitanski and Laptev cannot be larger than $1/4$, which is the optimal constant in the continuum. The proof of null-criticality rests on a new criterion for general graphs beyond the locally finite case. We also recover the situation of $\Z^d$ with $d \geq 3$ which can also be also treated by our new method.
\end{abstract}

\maketitle


\section{Introduction}
In recent years, there has been a growing interest in Hardy inequalities for discrete Laplacians. While Hardy's original inequality was on the discrete half-line $\N_0$, see \cite{HardyHistory} for a historical account, Hardy inequalities were overwhelmingly studied in the continuum setting, \cite{BEL,FLW23,KPS,Mazya,RS19}. There they have become a most fundamental tool in the analysis of partial differential equations, operator theory and mathematical physics. An outstanding and wide studied  feature of Hardy inequalities in the continuum is that in many cases,  the optimal constant can be determined and non-existence of minimizers can be shown.

First discrete Hardy inequalities for the standard Laplacian on the Euclidean lattice $\Z^d$, $d\ge2$, were proven by Rozenblum and Solomyak in \cite{RS09,RS13},  
and via a different method by Kapitanski and Laptev in \cite{KL16}. In these works, the authors show a Hardy inequality  with a weight of decay $|x|^{-2}$ as expected from the continuum setting. However, these methods do not allow to determine the optimal constant, and the issue of non-existence of minimizers is not addressed.

Another method, called the supersolution construction, was developed in \cite{DFP14,DP16} in the continuum setting, which was then transferred to graphs in \cite{KPP18}. This approach allows  to construct optimal Hardy weights from positive superharmonic functions, where optimal means that the Hardy weight cannot be improved, and the inequality does not admit a minimizer.
Specifically, in \cite{KPPHardy,KPP18}, an optimal Hardy weight for $\N_0$ and $\Z^d $, $d\ge3$, was constructed whose asymptotics resemble the optimal Hardy weight in the continuum. However, in contrast to the continuuum,  these discrete Hardy weight include higher order terms. In the one-dimensional case, these higher order terms are positive, see also~\cite{KS22,DFF}. In contrast, a  result of \cite{Gup23,gupta2024discrete} implicitly implies that at least in higher dimensions, the higher order terms must include negative terms as well.  

The main focus of this work is the two-dimensional case which is the most delicate one. While the main interest is the standard Laplacian $\Z^2$, our setting goes way beyond  to the entire class of aperiodic, symmetric random walks on $\Z^2$ with finite variance. For each such random walk, we construct an optimal Hardy weight  on $\Z^2 \setminus \{0\}$.
In particular, we allow for random walks with infinite support, which are not covered by the setting of \cite{KPP18}.
 Under additional moment conditions, we further derive the first two terms in the asymptotics of the weight.  While the leading term is the expected one,  the next higher order term is purely negative.
  This phenomenon stands in contrast to the one-dimensional case, and also to the continuous two-dimensional case. Our result comes with the important insight that we obtain the  sharp constant $1/4$ near infinity, and the global constant in the inequality of Kapitanski and Laptev cannot be larger.

 For the proof, we build on a new criterion which  was proven in the doctoral thesis of the first named author, \cite{hake2025optimal}, which we  prove here in a slightly more general form. A recent work \cite{HKPIII} presents another method to prove optimality of Hardy weights on graphs for non-locally finite graphs. This approach also covers optimality for our random walks on $\Z^d$, $d\ge 3$. However, the criterion of \cite{HKPIII} does not cover the $\Z^2$ case. We also point out that
some moment conditions on the random walk in question are necessary, as can be seen from the case of the fractional Laplacian treated in \cite{HKP}.

\section{Main results}

In this section, we present the main results of this paper.
We first introduce the relevant notions for    random walks over $\Z^d$. For background information, we refer to the monographs \cite{Spitzer76,Woe00,LawlerLimic}.

We equip $\Z^d$ with the Euclidian norm $|x| = \sqrt{\langle x,x \rangle}$ and the standard scalar product  from $\R^d$.
Throughout this work, we consider a function $\mu\colon \Z^d \to [0,\infty)$ 
whose  support generates $\Z^d$ as an additive group, a property which is called {\em aperiodic} in the context of random walks. Furthermore, we assume that $\mu$ is symmetric, i.e., $$\mu(x) = \mu(-x)$$ for all $x \in \Z^d$ and that $\mu$ has a finite second moment  $$\sum_{x \in \Z^d} \mu(x)|x|^2 < \infty. $$

We will call an aperiodic, symmetric $\mu$ with finite second moment a {\em {\rw}.} If $\mu$ gives rise to a probability measure, i.e.,    $\sum_{x\in \Z^d}\mu(x)=1$,  we will call $\mu$ a {\em  random walk}. We do not restrict ourselves to  random walks since our results are of a pure analytical nature and do not depend on the normalization. On the other hand, in our proofs we draw on results  which are formulated in a probabilistic language which we therefore also adopt here for {\rw}s. Clearly, every {\rw} can be normalized to a random walk by scaling.

Given $M > 0$, we write $$\|\mu\|_M = \sum_{x \in \Z^d} \mu(x)|x|^M$$ for the {\em $M$-th moment} of $\mu$, and say that $\mu$ {\em has finite $M$-th moment} if $\|\mu\|_M < \infty$. By definition, every {\rw} with finite $M$-th moment for $M \geq 2$,  has also finite $L$-th moment for all $0 < L \leq M$.

So, clearly existence of the finite second moment and symmetry, implies that the first moment vanishes, i.e.,
$$\sum_{x \in \Z^d} \mu(x) x =0$$
which is refered to as \emph{zero mean} for random walks.

We denote the {\em covariance matrix} for $\mu$
by
 $$\Sigma_{\mu} = \sum_{x \in \Z^d} \mu(x)xx^T$$
which is   the unique symmetric matrix $Q \in \R^{d \times d}$ such that 
\(
\langle x,\, Qy \rangle =  \sum_{x \in \Z^d} \mu(x)\langle x,y \rangle^2 \) { for all } \(y  \in \R^d.
\) 
We further define the linear transformation
$$x_{\mu} = ({\det \Sigma_{\mu}})^{\frac{1}{2d}} \, \big(\Sigma_{\mu}^{-1/2}x\big), \quad x \in \R^d,$$
where $\Sigma_{\mu}^{-1/2}$ is defined via the spectral calculus, making use of the fact that $\Sigma_{\mu}$ is symmetric and positive definite. We further set $|x|_{\mu}= |x_{\mu}|$ for $x \in \R^d$.  We point out that for multiples $\mu$ of the simple random walk represented by the measure $ (2d)^{-1} 1_{|x|=1}$, we have $|x| = |x|_{\mu}$, $x\in\Z^d$. More generally, this is the case if 
$\vartheta_{\mu}^2(x) = ({\det \Sigma_{\mu}})^{\frac{1}{d}}$ for all $x \in \Z^d$, $x \neq 0$, where $\vartheta_{\mu}^2$ denotes  the normalized covariance form 
\[
\vartheta^2_{\mu}\colon \Z^d \setminus \{0\} \to [0, \infty), \qquad \vartheta^2_{\mu}(x) = \frac{\langle x,\Sigma_{\mu}x \rangle}{\langle x,x \rangle}.
\] 
 
 \medskip
 
 We next discuss {\em graphs} over a countable discrete set $X$ which are symmetric functions
$b\colon X\times X \to [0, \infty)$  with vanishing diagonal and which are {\em locally summable}, i.e., $\sum_{y \in X}b(x,y) < \infty$ about every $x \in X$. The {\em formal Laplacian}  $\mathcal{L}\colon \mathcal{F} \to C(X)$ is then defined as
\[
\mathcal{L}u(x) = \sum_{y \in X} b(x,y)\big( u(x) - u(y) \big),
\]
where $C(X)$ denotes the vector space of real-valued functions $f\colon X \to \R$, and $\mathcal{F} = \{f \in C(X)\mid \sum_{y \in X} b(x,y)|f(y)| < \infty \mbox{ for all } x \in X \}$. 
The {\em bilinear form}  $\mathcal{Q}\colon C_c(X) \times C_c(X) \to \R$ associated with $b$ is given as 
\[
\mathcal{Q}\big( \varphi, \psi \big) = \frac{1}{2} \sum_{x,y \in\Z^d} b(x,y) \big( \varphi(x) - \varphi(y) \big) \big( \psi(x) - \psi(y) \big),
\] 
where $C_c(X) \subseteq C(X)$ is the subspace of functions of finite support. We further denote $\mathcal{Q}(\varphi) := \mathcal{Q}(\varphi, \varphi)$ for $\varphi \in C_c(X)$.

For a subset $Y \subseteq X$, a non-trivial function $w:Y \to [0,\infty)$ is called a {\em Hardy weight} on $Y$ if for all $\varphi \in C_c(X)$ vanishing outside of $Y$, we have validity of the {\em Hardy inequality}
\begin{align*} 
\mathcal{Q}(\varphi) \geq w(\varphi) := \sum_{x \in Y}w(x) \varphi^2(x) . 
\end{align*}
A Hardy weight $w$ is called {\em critical} if for any Hardy weight $w' \geq w$, one has $w' = w$. In this case, there exists a unique (up to scalar multiples)  function $u$ which is positive on $Y$, vanishing outside of $Y$ and satisfies  $\mathcal{L}u = wu$ on $Y$. This $u$ is called the {\em (Agmon) ground state} for $w$. A critical Hardy weight $w$ is called {\em null-critical} if the ground state $u$ is not square integrable with respect to the measure $w$, i.e.,\@ $\sum_{x \in Y} w(x) u^2(x) = \infty$. This property can be interpreted as saying that the Hardy inequality 
does not admit a minimizer in the closure of $C_c(X)$ with respect to form and pointwise convergence. A critical Hardy weight which is not null-critical is called {\em positive-critical}. 
Null-critical Hardy weights are also {\em optimal near infinity}, i.e.,\@ if $(1+\lambda)w$ is a Hardy weight on $Y\setminus K$ for some finite $K$ and $\lambda\ge0$, then $\lambda=0$. In other words, there cannot be a Hardy weight dominating $w$ near infinity.
For a proof of this fact, see e.g.,\@ \cite[Proposition~15]{HKPII} as well as \cite{Fischer,KN,KPP18,KovarikPinchover}. Null-critical Hardy weights are also referred to as {\em optimal} Hardy weights, \cite{DFP14,DP16}.

We point out that a {\rw} $\mu$ on $X=\Z^d$ gives rise to a graph $b_{\mu}$ over $\Z^d$ via $$b_{\mu}(x,y) = \mu(x-y)$$ for $x \neq y$ and $b_{\mu}(x,x) =0$, where $x,y \in \Z^d$.
 We write $\Delta_{\mu}=\mathcal{L} $ for the formal Laplacian in this case. 

 \medskip
 
 Let us now specify to $d=2$.
The simple random walk on $\Z^2$ is recurrent, and so are random walks with finite variance, i.e., finite second moment, see \cite[Section~8, T1~(b)]{Spitzer76}. It can be inferred from criticality theory that in these situations, there cannot be   a Hardy weight for $Y=\Z^2$, see \cite[Theorem~6.1]{KLW} or \cite{F00}. However, adding a Dirichlet boundary condition for $\Delta_{\mu}$ at $x=0$ gives rise to a transient graph in the sense of \cite[Chapter~6]{KLW}. Since transience of a graph is equivalent to the existence of Hardy weights, cf.\@ \cite[Theorem~6.1]{KLW}, we can study Hardy weights along with their optimality criteria for $Y=\Z^2 \setminus \{0\}$. 
This is where we are heading next. 

\medskip

To this end, we consider the $d$-dimensional torus $\mathbb{T}^d = [-\pi,\pi]^d$, endowed with the Lebesgue measure $d\xi$. The {\em potential kernel} for a {\rw} $\mu$ on $\Z^2$ is given as 
\[
a_{\mu}:\Z^2 \to [0,\infty), \quad a_{\mu}(x) = \frac{1}{(2\pi)^2} \int_{\mathbb{T}^2} \frac{1-\exp(-i\langle x, \xi \rangle)}{1-\widehat{\mu}_0(\xi)}\, d\xi,
\]
where  $\widehat{\mu}_0$ denotes the characteristic function of the random walk $\mu_0 = \mu/\sum_{x \in \Z^2} \mu(x)$, i.e.,\@
\[
\widehat{\mu}_0(\xi)  =\sum_{x \in \Z^2} \mu_0(x)\exp(-i\langle x, \xi \rangle).
\]
Our  main result concerns the construction of an optimal, i.e., null-critical Hardy weight for {\rw}s over $\Z^2$.

\begin{theorem} \label{thm:MAIN}
	Let $\mu$ be a   {\rw} over $\Z^2$. Then, 
	\[
	w_{\mu}: \Z^2 \setminus \{0\} \to [0, \infty), \quad w_{\mu}(x) = \frac{\Delta_{\mu} a^{1/2}_{\mu}(x)}{a^{1/2}_{\mu}(x)}
	\]
	is a null-critical Hardy weight  on $\Z^2 \setminus  \{0\}$. If in addition, $\|\mu\|_{6}< \infty$, then as $|x| \to \infty$,  
	\begin{multline*}
		 w_{\mu}(x) =  \, \frac{\sqrt{\det \Sigma_{\mu}}}{8} \cdot \frac{1}{|x|_{\mu}^2 (\log|x|_{\mu})^2} \, - \, \frac{\kappa_{\mu} \pi \det \Sigma_{\mu}}{4\theta} \frac{1}{|x|_{\mu}^2 (\log |x|_{\mu})^3}  \\
		  + \, \mathcal{O}\left( \frac{1}{|x|^2(\log|x|)^4} \right), 
	\end{multline*}  
	where $\theta=\sum_{v \in \Z^2}\mu(v)$ and $\kappa_{\mu} $ is a constant.
\end{theorem}

\begin{remark}
The constant $\kappa_{\mu}$ appears in the asymptotic expansion of the potential kernel $a_{\mu}$ proven in \cite{FU96}. It is explicit for the simple random walk as we discuss below. While the authors of \cite{FU96} do not give an explicit formula for $\kappa_{\mu}$, they refer to \cite[Section~12, P3]{Spitzer76} who proves the expansion in a slightly less general case. There, $\kappa_{\mu}$ is a sum of positive constants which involve Euler's constant $\gamma$, Catalan's constant, $\pi$ and a positive integral and, hence, $\kappa_\mu>0$. The corresponding integral in the more general case of \cite{FU96} can then be derived by a change of variables.
\end{remark}

\medskip
Specifying our result to the standard Laplacian on $\Z^2$, i.e., the simple random walk, the theorem yields the following optimal Hardy inequality. 

\begin{corollary} \label{cor:MAIN}
For the {\rw}  $\mu = 1_{|x|=1}$ over $\Z^2$, the null-critical Hardy weight   $w = ({\Delta_{\mu} a_{\mu}^{1/2}})/{a_{\mu}^{1/2}}$  on $\Z^2 \setminus \{0\}$  satisfies
	\begin{align*}
		w(x)  = \frac{1}{4} \frac{1}{|x|^2 (\log |x|)^2} - \frac{\log 8 + 2\gamma}{4} \frac{1}{|x|^2 (\log|x|)^3} + \mathcal{O}\Big( \frac{1}{|x|^2 (\log|x|)^4} \Big)
	\end{align*}
	as $|x| \to \infty$, where $\gamma > 0$ is the Euler constant.
\end{corollary}

\begin{remark}
 Kapitanski and Laptev prove in \cite[Theorem~4.1]{KL16} a Hardy inequality with a weight $w_{KL}$ in  $\Z^2 \setminus \{x\mid |x| \leq 1\}$, defined as
 $$w_{KL}(x) = c\,|x|^{-2} (\log|x|)^{-2}$$ for some non-explicit constant $c > 0$. 
 Clearly, their result is an immediate consequence of the above corollary which  gives more information in various aspects. A minor aspect   is that our result holds for all $\varphi$ vanishing at $0$, while the inequality in \cite[Theorem~4.1]{KL16} requires $\varphi$ to vanish for all $x \in \Z^2 $ with $|x| \leq 1$. In \cite[Proposition~4.2]{RS13}  Rozenblum and Solomyak dealt with this issue by working with a weight of the form 
 \[
 w_{RS}(x) = c\,|x|^{-2}\big( \log(|x| + 2) \big)^{-2}
 \]  
 instead. 
 A more important aspect of our result is that  the above defined $w$  is critical,  so there cannot be a weight dominating $w$. Since $w$ is null-critical and, hence, optimal near infinity, 
the constant $1/4$ in front of the leading term of the expansion is sharp, and, thus, it is the optimal constant for the Hardy inequality near infinity. It further coincides with the best possible constant of the continous counterpart of the above inequality in $\R^2$ which reads as 
\[
 \int_{\{ |x| > 1\}} \big| \nabla u(x) \big|^2\, dx \geq \frac{1}{4} \int_{\{|x| > 1\}} \frac{u(x)^2}{|x|^2 (\log|x|)^2}\, dx 
\]
for sufficiently smooth functions $u$, see \cite[Section~1.3.1]{Mazya} or \cite[Proposition 2.76]{FLW23}. However, the  next higher order error terms in our expansion is negative. This indicates  that the Hardy inequality of Kapitanski and Laptev may not even hold for $c=1/4$ in contrast to the situation in $\R^2$, where $w(x)=|x|^{-2}(\log|x|)^{-2}/4$ is optimal, cf.~\cite[Example 13.1]{DFP14}. The situation in $\Z^2$ is 
 opposite to the one-dimensional case of $\N_0$, where the higher order terms in the expansion of the optimal Hardy weight are strictly positive, see \cite{KPP18,KPPHardy}.
\end{remark}

Let us comment on the proof of Theorem~\ref{thm:MAIN}. For the expansion of the weight under the stronger moment condition on $\mu$, we use 
the fine  asymptotics of the potential kernel $a_{\mu}$ derived in \cite{FU96,Uch98}, see Theorem~\ref{thm:UchiyamaFukai} below. The formula describing $w_{\mu}$ via the harmonic function $a_{\mu}$
is called the {\em supersolution construction}. The latter has been developed in the fundamental works \cite{DFP14,DP16} in the continuum setting, and has been established for graphs in \cite{KPP18,HKPIII}, see also \cite{KL23}. However, none of the criteria provided in the latter works cover the situation of the above theorem. For instance, the potential kernel 
 does not arise by convolution of the Green kernel, and consequently, the criteria from  \cite{HKPIII} do not apply. Moreover, the method from \cite{KPP18} is bound to locally finite graphs and, thus, does not include {\rw}s with infinite support. The criterion developed in this paper is a generalization of a result from \cite{KPP18} to non-locally finite graphs. 
 
 \medskip
 
 The following theorem is the key tool to prove null-criticality of the aforementioned Hardy inequalities on $\Z^2 \setminus \{0\}$.  Here, we say that a function $u:X \to [0,\infty)$ is {\em proper} if $u^{-1}(I)$ is a finite set for every compact interval $I \subseteq (0,\infty)$. We say that $Y \subseteq X$ is a {\em proper subset} if $Y \neq X$.

\begin{theorem} \label{thm:CRITERION} 
	Let $b$ be a  connected graph over $X$ and let $K \subseteq X$ be a proper and non-empty subset.
	Suppose that $u \in \mathcal{F}$ is strictly positive on $X \setminus K$, $u=0$ on $K$ and $\mathcal{L}u = 0$ on $X \setminus K$. If $u$ is additionally proper and satisfies a weak bounded oscillation condition
	\[
	\sum_{x,y \in X,\, u(y)< t_0 \leq u(x)} b(x,y)\big( u(x) - u(y) \big) < \infty
	\]
	for some $t_0 > 0$,
	then $w:X \setminus K \to [0, \infty)$ given as
	\[
	w = \frac{\mathcal{L} u^{1/2}}{u^{1/2}}
	\]
	is a null-critical Hardy weight on $X \setminus K$. 
\end{theorem}

The above theorem  generalizes \cite[Theorem~1.1]{KPP18} in that it allows for potentially  non-locally finite graphs  and the set $K$ does not need to be finite. We obtain that $w_{\mu}$ as defined above is a null-critical Hardy weight on $\Z^2 \setminus \{0\}$ by applying the theorem with $X=\Z^2$, $b=b_{\mu}$, $K= \{0\}$ and $u=a_{\mu}$. 
We will derive Theorem~\ref{thm:CRITERION}  from another general criterion given in Theorem~\ref{thm:abstractmain}, which is of independent interest. For instance, the latter can also be used for {\rw}s on $\Z^d$ beyond $d=2$.

\medskip

Let us briefly discuss the situation when $d\geq 3$. We will use the aformentioned Theorem~\ref{thm:abstractmain} in order to show that the supersolution construction applied to the Green's function  yields optimal Hardy weights. In contrast to the case $d=2$, this can also be deduced from methods building on potential theory, cf.\@ \cite{HKPIII}. For $d \geq 3$, consider 
\begin{align*}
	G_{\mu}:\Z^d \to (0, \infty), \quad G_{\mu}(x) = \frac{1}{(2\pi)^d} \int_{\mathbb{T}^d} \frac{\exp\big( -i\langle x,\xi \rangle \big)}{1-\widehat{\mu}_0(\xi)}\,d\xi,
\end{align*}
where as above, $\widehat{\mu}_0$ is the characteristic function of the  random walk $\mu_0= \mu/\sum_{x \in \Z^d} \mu(x)$. One can verify that $G_\mu$ is the Fourier representation of the Green function of the random walk $\mu_0$ and, thus, a positive multiple of the Green function of $\Delta_\mu$.

\begin{theorem} \label{thm:MAIN2}
	Let $\mu$ be a  {\rw} on $\Z^d$, $d \geq 3$. Then 
	\begin{align*}
		w_{\mu}:\Z^d \to [0,\infty), \quad w_{\mu}(x) = \frac{\Delta_{\mu} G_{\mu}^{1/2}(x)}{G_{\mu}^{1/2}(x)}
	\end{align*}
	defines a null-critical Hardy weight on $\Z^d$. If additionally $\|\mu\|_{d+3} < \infty$, then as $|x| \to \infty$,  
	\begin{align*}
		w_{\mu}(x) = \frac{(d-2)^2 (\det \Sigma_{\mu})^{\frac{1}{d}}}{8|x|_{\mu}^2} + \mathcal{O}\Big( |x|^{-4} \Big). 
	\end{align*}	
\end{theorem}
\begin{remark}
For the Laplacian with standard weights, i.e., $\mu=1_{|x|=1}$, this result has already been known with higher order terms  $\mathcal{O}(|x|^{-3})$, cf.\@ \cite[Theorem~7.2]{KPP18}. The major novelty of the above result is that it holds for {\rw}s with infinite support, i.e., non-locally finite graphs over $\Z^d$, as well, which was not covered by the methods of \cite{KPP18}.
 
Here, we prove the corresponding asymptotic behavior under a moment condition using the fine asymptotics on $G_{\mu}$ provided in \cite{Uch98}.
By looking at the fractional Laplacian it becomes clear that some moment condition is necessary: In \cite{HKP,HKPIII} it is shown, that for the fractional Laplacian $(-\Delta)^{s}$ on $\Z^d$ with $s \in (0,1)$, optimal Hardy weights decay as $|x|^{-2s}$ which is clearly slower than $|x|^{-2}$. Furthermore, the {\rw} decays as $|x|^{-d-2s}$, cf.\@ \cite[Theorem~A.1]{HKP}. Hence, these transition weights have finite $M$-th moment if and only if $M < 2s$.	
\end{remark}	 
\medskip

The paper is organized as follows. In the next section we prove the  criterion for general graphs,  Theorem~\ref{thm:CRITERION}. The latter is derived from Theorem~\ref{thm:abstractmain}, which is a slight generalization of a result from the doctoral dissertation thesis of the first named author \cite[Theorem~2.26]{hake2025optimal}. We then prove the claimed expansions for the weight $w_{\mu}$, see Section~\ref{sec:proofdge3} for $d \geq 3$, and Section~\ref{sec:proofd=2} for $d=2$.  
In the final Section~\ref{sec:proofs} we combine the criteria from Section~\ref{sec:criterion} with the obtained expansions, thus establishing Theorem~\ref{thm:MAIN}, Corollary~\ref{cor:MAIN}, and Theorem~\ref{thm:MAIN2}.

\section{Criterion for general graphs} \label{sec:criterion}

In this section we prove two related criteria for general graphs 
to show optimal Hardy weights via the supersolution construction. 

\medskip

To this end, let us first clarify some relevant notions. We extend the concept of a graph given in the introduction by positive potentials. Namely, a pair $(b,c)$ is called a {\em graph} over a discrete, countable set $X$ if $b:X\times X\to [0,\infty)$ is symmetric, has zero diagonal and is locally summable as before and $c:X \to [0,\infty)$ is an arbitrary function, called the {\em killing term}. The elements in $X$ are called the {\em vertices} of $(b,c)$ and two vertices $x,y \in X$ are connected by an {\em edge} if $b(x,y) > 0$.
In case that $c=0$ is the zero function, we simple call $b$ instead of $(b,0)$ a graph over $X$. A graph is said to be {\em connected} if for all $x,y \in X$, there are $x=y_1, y_2, \dots, y_n = y$ such that $b(y_{i}, y_{i+1}) > 0$ for all $1 \leq i \leq n-1$.  From now on, we will always assume that we deal with a    connected graph $(b,c)$ over $X$.

Every graph gives rise to the {\em formal Laplacian} $ \mathcal{L}:\mathcal{F} \to C(X), $
\[
\mathcal{L}u(x) = \sum_{y \in X} b(x,y)\big( u(x) - u(y) \big) \, + \, c(x)u(x),
\]
with $\mathcal{F}$ and $C(X)$ as defined before. 
Given $Y \subseteq X$, we say that a function $u:X \to \R$ is {\em strictly positive on $Y$} and we write $u > 0$ on $Y$ if $u(x) > 0$ for all $x \in Y$. Further, we say that $u \in \mathcal{F}$ is {\em superharmonic}, (respectively {\em harmonic}) on $Y$, if $\mathcal{L}u \geq 0$, (respectively $\mathcal{L}u = 0$) on $ Y$. If $X=Y$, we just say that $u$ is {\em superharmonic}, (respectively {\em harmonic}).

The  bilinear form $\mathcal{Q}: C_c(X) \times C_c(X) \to \R $ associated to $ \mathcal{L} $ then acts as  
\[
  \mathcal{Q}( \varphi, \psi ) = \frac{1}{2} \sum_{x,y \in X} b(x,y) \big( \varphi(x) - \varphi(y) \big) \big( \psi(x) - \psi(y) \big) \, + \, \sum_{x \in X} c(x)\varphi(x)\psi(x).
\] 
The {\em extended space} $\mathcal{D}_0 = \mathcal{D}_0(\mathcal{Q})$ associated with a quadratic form $\mathcal{Q}$ is the closure of $C_c(X)$ in $\mathcal{D}$
with respect to the norm $\|\cdot\|_{\mathcal{Q},o}$, where $o \in X$ is a fixed vertex in the connected graph $(b,c)$ and $\|\varphi\|^2_{\mathcal{Q},o} = \mathcal{Q}(\varphi) + |\varphi(o)|^2$ and $\mathcal{Q}(\varphi) = \mathcal{Q}(\varphi,\varphi)$. 
  
 Given a non-negative function $w\colon X \to [0, \infty)$ and $p \in [1, \infty)$, we write $\ell^p(X,w)$ for the subspace of all $f \in C(X)$ with $\sum_{ X}w  |f|^p < \infty$, where for  sums over subsets $Y \subseteq X$, we  use the short hand notation
\[
\sum_Y f := \sum_{y \in Y} f(y),
\]
whenever $f \in C(X)$ is such that the sum on the right hand side is absolutely convergent or all terms are non-negative. Furthermore, we write $\ell^p(X) = \ell^p(X,1)$.

\medskip
A non-negative function $w\colon X \to [0, \infty)$ is called a {\em Hardy weight} on $Y \subseteq X$ if it satisfies a {\em Hardy inequality} on $Y$, i.e., $
	\mathcal{Q}(\varphi) \, \geq \, w(\varphi)$
for all $\varphi \in C_c(X)$ with $\varphi(x) = 0$ for $x \in X \setminus Y$.  Non-trivial Hardy weights on $Y=X$ exist precisely for transient graphs, see \cite[Chapter~6]{KLW} for background information. In this case, the Hardy inequality can be extended to all of $\mathcal{D}_0$. Note that also under the presence of a non-trivial killing term, the notions of {\em critical} and {\em null-critical} Hardy weights can be defined completely analogously as done in the previous section.

For a function $ u\in \mathcal{F} $, we define the function $ g:[0,\infty)\to [0, \infty]$  as
\[
g_{u}(t) = \sum_{x,y\in X, u(y) < t \leq u(x)} b(x,y) \big( u(x) - u(y) \big) .
\]
It has been shown in \cite[Lemma~2.8]{KPP18} that for a strictly positive, proper $u \in \mathcal{F}$,  one has
\[
g_{u}(s) = g_{u}(t) + \sum_{x \in X,\, s < u(x) \leq t} \mathcal{L}u(x)
\]
for all $t > s >0$, where both sides may be infinite. However, note that the sum over $\mathcal{L}u$ on the right hand side is always finite due to properness of $u$.

We say  $u$ satisfies a {\em weak bounded oscillation condition}, if $g_{u}(t)$ is finite for some $ t>0$. Hence, if $u$ satisfies a weak bounded oscillation condition and  $\mathcal{L}u \in \ell^1(X)$, then  $g_{u}$ is bounded.

\medskip

In this section we prove the following general criterion for obtaining null-critical Hardy weights.  

\begin{theorem} \label{thm:abstractmain}
 	Let $u \in \mathcal{F}$, $u > 0$
	satisfy   the following conditions:
	\begin{itemize}
		\item $u$ is superharmonic,
		\item $\mathcal{L}u \in \ell^1(X)$,
		\item $u$ is proper and satisfies a weak bounded oscillation condition.
	\end{itemize}
	Then, 
	$w= {\mathcal{L} u^{1/2}}/{u^{1/2}}
	$
	is a critical Hardy weight on $X$. If additionally,
	\begin{itemize}
		\item $X$ is infinite,
		\item $\sum_X \mathcal{L}u \neq \sum_X cu$ \, or \, $c=0$, 
	\end{itemize}
	then $w$ is null-critical.
\end{theorem}

\begin{remark}
	We recall from \cite{HKPIII}, that if $u$ is a superharmonic potential, i.e., the image of a positive function of the Green operator, and if also  $\mathcal{L}u \in \ell^1(X)$, then $w$ is critical. In this situation, a weak bounded oscillation condition is necessarily fulfilled, cf.\@ \cite[Remark~2.27~(4)]{hake2025optimal}. Given a strictly positive, superharmonic function $u$ with integrable Laplacian, the above theorem provides an alternative criterion for obtaining a critical Hardy weight via the supersolution construction by assuming in addition that $u$ is proper and satisfies a weak bounded oscillation condition. 
\end{remark}

Before we turn to the proof of the theorem, we show that any non-negative $u \in \mathcal{F}$ that vanishes at infinity or that tends to infinity 
 is proper and satisfies the assumed weak bounded oscillation condition for arbitrary $t > 0$.  Here, we say that $u \in C(X)$ {\em vanishes at infinity} and we write $u \in C_0(X)$ if for each $\varepsilon > 0$, there is a finite set $K \subseteq X$ such that $|u|< \varepsilon$ on $ X \setminus K$.
We say that $u$ {\em tends to infinity} if for each $C >0$ there exists a finite set $K \subseteq X$ such that $|u| > C$ on $X \setminus K$. We denote $I_{\infty}(X)$ as the set of all $u \in C(X)$ that tend to infinity.

\begin{proposition} \label{prop:criterionC0}
 	Assume that $u \in ( C_0(X) \cup I_{\infty}(X) ) \cap \mathcal{F}$ is  non-negative. Then $u$ is proper and satisfies a weak bounded oscillation condition.
\end{proposition}

\begin{proof}
	We carry out the proof for $u \in C_0(X) \cap \mathcal{F}$, the case $u \in I_{\infty}(X) \cap \mathcal{F}$ can be established along similar lines. Clearly, each $u \in C_0(X)$ is proper. 
	 Fix $t > 0$. Since $u \in C_0(X)$, the level set $\Omega_t = \{x \in X\mid t < u(x) \}$ is finite. We estimate
	\begin{align*}
		g_{u}(t) &\leq \sum_{x \in \Omega_t} \sum_{y\in X\setminus\Omega_t} b(x,y) u(x) \, + \, \sum_{x \in \Omega_t} \sum_{y\in X\setminus\Omega_t} b(x,y) u(y).
	\end{align*}
	The first double sum is clearly finite since $\Omega_t$ is finite and the graph is locally summable. For the second double sum, note that for each $x \in \Omega_t$, 
	\begin{align*}
		\sum_{y\in X\setminus\Omega_t} b(x,y) u(y) \leq  t \sum_{y \in X} b(x,y) < \infty.
	\end{align*} 
	This finishes the proof. 	
\end{proof}

The remainder of this section is devoted to the proof of  Theorem~\ref{thm:abstractmain} above and the proof of Theorem~\ref{thm:CRITERION} from the introduction.

\medskip

For $T > 1$, define the functions $\lambda_T^+, \lambda_T^-:\R \rightarrow \R$ by
\begin{align*}
	\lambda_T^+(t) = \int_{\frac{1}{T}}^{1} \frac{1_{[0,t]}(s)}{s} ds, \quad \lambda_T^-(t) = \int_1^{T} \frac{1_{[0,t]}(s)}{s} ds,
\end{align*}
as well as $$\lambda_T = \lambda_T^{+} - \lambda_T^{-}.$$ 
Given $u \in \mathcal{F}$, we write $\lambda_T(u) = \lambda_T \circ u$ and for   a  positive $v \in \mathcal{F}$, we define the quadratic form on $C_c(X)$ by
\[
\mathcal{Q}_v(\varphi) = \frac{1}{2} \sum_{x,y \in X} b(x,y)v(x)v(y)\big( \varphi(x) - \varphi(y) \big)^2.
\] 

\begin{lemma} \label{lemma:key}
	Let $u \in \mathcal{F}$ be strictly positive,  proper, satisfy a weak bounded oscillation condition, superharmonic with $\mathcal{L}u \in \ell^1(X)$. Then, $ g_{u} $ is bounded, $\lambda_T(u) \in \mathcal{D}_0(\mathcal{Q}_{u^{1/2}})$ and $\mathcal{Q}_{u^{1/2}}\big( \lambda_T(u) \big) \leq 4 \log T \sup_{[0,\infty)} |g_u|$
	for each $T >1$.
\end{lemma}

\begin{proof}
	We point out that boundedness  of $ g_{u} $ is immediate from the weak bounded oscillation condition, $ \mathcal{L}u\in \ell^1(X) $ 	and	the remarks made above Theorem~\ref{thm:abstractmain}. 
	We abbreviate $v=u^{1/2}$. 
	Take $T > 1$. Observe first that if $t < 1/T$, then $\lambda_T^{+}(t) = \lambda_T^{-}(t) = 0$ and if $t > T$, then $\lambda_T^{+}(t) = \lambda_T^{-}(t) = \log T$. So, $\lambda_T(t) = 0$ in either case. Since $u$ is proper, we have $\lambda_T(u) \in C_c(X) \subseteq \mathcal{D}_0(\mathcal{Q}_v)$.  
	Using \cite[Lemma~10]{HKPIII}, we compute 
	\begin{multline*}
	 	2v(x)v(y) \big| \lambda_T^{\pm}(u)(x) - \lambda_T^{\pm}(u)(x)  \big|^2 \\ 
		 \leq |u(x)-u(y)| \big| \lambda_T^{\pm}(u)(x) - \lambda_T^{\pm}(u)(x)  \big| \\
		 = \big( u(x)-u(y) \big) \big(  \lambda_T^{\pm}(u)(x) - \lambda_T^{\pm}(u)(x)  \big),
	\end{multline*}
	where for the last equality we used that $\lambda_{T}^{\pm}$ is monotonically increasing. We further have $u \lambda_{T}^{\pm}(u) \geq 0$ since $\lambda_{T}^{\pm}$ is monotonically increasing and $\lambda_T^{\pm}(0) =0$. Thus, with  all summands being non-negative, the values
	$\mathcal{Q}(u,\lambda_{T}^{\pm}(u))$ can be formally defined (and apriori might be infinite). We obtain
	\begin{align*}
		2\mathcal{Q}_{v} \big( \lambda_T^{\pm}(u) \big) &\leq \frac{1}{2} \sum_{x,y \in X} b(x,y) \big( u(x) - u(y) \big) \big( \lambda_T^{\pm}(u)(x) - \lambda_T^{\pm}(u)(x) \big) \\
		&= \mathcal{Q}(u, \lambda_T^{\pm}(u)). 
	\end{align*}
	Consequently, 
	\begin{multline*}
		\mathcal{Q}_v\big( \lambda_T(u) \big) \leq 2\mathcal{Q}_v\big( \lambda_T^{+}(u) \big) + 2\mathcal{Q}_v\big( \lambda_T^{-}(u) \big) \\
		\leq  \mathcal{Q}\big(u, \lambda_T^{+}(u) \big) + \mathcal{Q}\big(u, \lambda_T^{-}(u) \big) 
		= \mathcal{Q}\big(u, \lambda_T^{*} (u) \big)
	\end{multline*}
	with $\lambda_T^{*} = \lambda_T^{+} + \lambda_T^{-}$. 
	Moreover, if $u(x)> u(y)$, then 
	\[
	\lambda_T^{*}(u)(x) - \lambda_T^{*}(u)(y) = \int_{1/T}^T \frac{1_{(u(y),u(x)]}(t)}{t}\,dt.
	\]
	By Fubini's theorem, 
	\begin{align*}
		\mathcal{Q}\big( u, \lambda_T^{*}(u)\big) &= \sum_{x,y\in X, u(x)>u(y)} b(x,y) \big( u(x) - u(y) \big)\big( \lambda_T^{*}(u)(x) - \lambda_T^{*}(u)(y) \big) \\
		&= \int_{1/T}^T \frac{1}{t} \sum_{x,y\in X, u(x)> u(y)} b(x,y) \big( u(x) - u(y) \big) 1_{(u(y),u(x)]}(t)\, dt \\
		&= \int_{1/T}^T \frac{g_{u}(t)}{t}\,dt \\
		&\leq 2 \log T \sup_{[0,\infty)} |g_u|. 
	\end{align*}
	This finishes the proof.  
\end{proof}

Using a result from \cite{HKPIII}, we can now show that the weight $w$ in Theorem~\ref{thm:abstractmain} is a critical Hardy weight. 

\begin{proposition}
Let $u \in \mathcal{F}$ be strictly positive,  proper,  superharmonic with $\mathcal{L}u \in \ell^1(X)$ and satisfy a weak bounded oscillation condition. Then, $w=\mathcal{L}u^{1/2} / u^{1/2}$ is a critical Hardy weight.
\end{proposition}

\begin{proof}
	Since $u$ is superharmonic, so is $v=u^{1/2}$, cf.~\cite[Corollary 2.3]{KPP18} and, thus, $w$ is a Hardy weight. 
	It follows from \cite[Proposition~14]{HKPIII} that 
	 $w$ is critical if $\lambda_T(u) \in \mathcal{D}_0(\mathcal{Q}_{v})$ and there is some $C> 0$ such that $\mathcal{Q}_{v}\big( \lambda_T(u) \big) \leq C \log T$ for each $T > 1$. These assumptions in turn are satisfied due to Lemma~\ref{lemma:key}.  
\end{proof}

We are now in position to prove the main theorem of this section.

\begin{proof}[Proof of Theorem~\ref{thm:abstractmain}]
	The previous proposition already yields that $w$ is a critical Hardy weight. We will now make use of the remaining assumptions, i.e., in the second list of bullet points in Theorem~\ref{thm:abstractmain}. 
	
	First, observe for $c\neq0$, that $\sum_{X} \mathcal{L}u = \sum_{X} cu$ is a necessary  condition 
	for positive-criticality by \cite[Proposition~21 and Proposition~22]{HKPII}, so $\sum_{X} \mathcal{L}u \neq \sum_{X} cu$ is a sufficient  condition 
	for null-criticality. 
	
	It remains to prove null-criticality for inifinite graphs where $c=0$ and $\sum_X \mathcal{L}u =0$. Since $u$ is superharmonic, the latter condition implies that $u$ is actually harmonic. Consequently, the function $g_u$  is constant.	Since $X$ is infinite and $u$ is proper, $u$ is non-constant, and so, $g_u > 0$. Suppose now by contradiction that $w$ is not null-critical, i.e., positive-critical. Then $v= u^{1/2} \in \mathcal{D}_0$ by \cite[Proposition~21]{HKPII} and by the variant of the coarea formula given in \cite[Theorem~2.6]{KPP18}, we obtain 
	\begin{align*}
		\int_0^{\infty} g_v(t)\,dt \, = \, \mathcal{Q}(v) < \infty.
	\end{align*} 
	It follows from this that $\liminf_{t \to \infty} t g_v(t) = 0$ and the dominated convergence theorem yields
	\[
	\lim_{t \to \infty} \sum_{x,y\in X, v(y) < t \leq v(x)} b(x,y)\big( v(x)-v(y) \big)^2 \, = \, 0.
	\]
	Consider now $x,y \in X$ along with $t > 0$ such that $v(y) < t \leq v(x)$. If $v(y) \leq (1/2)v(x)$, then 
	\[
	\frac{u(x)-u(y)}{\big( v(x)-v(y) \big)^2} = \frac{v(x)+ v(y)}{v(x)-v(y)} \leq 3,
	\]
	and if $v(y) > (1/2)v(x)$, then it follows from the mean value theorem that 
	\[
	\frac{v(x)-v(y)}{u(x)-u(y)} \geq \frac{1}{2v(x)} \geq \frac{1}{4v(y)} \geq \frac{1}{4t}.	
	\]
	In either case, 
	\[
	u(x)- u(y) \leq 4t\big( v(x)-v(y) \big) + 3 \big(v(x)-v(y) \big)^2.
	\]
	So for each $t > 0$, we have 
	\begin{align*}
		g_u(t^2) &= \sum_{x,y\in X, u(y) <t^2 \leq u(x)} b(x,y)\big( u(x)-u(y) \big) \\
		&\leq \sum_{x,y \in X,v(y) <t \leq v(x)} b(x,y)\left( 4t\big( v(x)-v(y) \big) + 3 \big(v(x)-v(y) \big)^2  \right) \\
		&= 4tg_v(t) \, + \, 3 \sum_{x,y\in X, v(y) <t \leq v(x)}b(x,y)\big(v(x)-v(y) \big)^2. 
	\end{align*}
	The previous considerations imply $\liminf_{t \to \infty} g_u(t^2) =0$. However, this is impossible since $g_u$ was shown to be constant and strictly positive. 
\end{proof}

We can now prove the abstract criterion formulated in the introduction.

\begin{proof}[Proof of Theorem~\ref{thm:CRITERION}]
	We observe first that every connected component of the restriction of the graph to $X \setminus K$ is an infinite set.
	 Indeed, since $K \neq X$, the restriction of $u$ on a connected component $Y$ of $X \setminus K$ is a non-trivial harmonic function for the Dirichlet restriction of $\mathcal{L}$ to $Y$ as defined in \cite[p.\@~118]{KLW}. However, as $K \neq \emptyset$, such a function cannot exist for finite $Y$, since the Dirichlet restriction on a proper finite subset of a connected graph has only positive eigenvalues, see Proposition~1.20~(a) and its proof in \cite{KLW}. Next, we define a killing term $c_K\colon X \setminus K \to [0, \infty)$ by setting $c_K(x) = \sum_{y \in K}b(x,y)$. 
	We denote the restriction of $b$ to $X\setminus K  \times X \setminus K$ by $b_K$. For the Laplacian $\mathcal{L}_K$ associated with 
	$(b_K,c_K)$, we have $\mathcal{L}_K u_{|X \setminus K} = \mathcal{L} u$ on $X \setminus K$, and in particular, $u_K= u_{|X \setminus K}$ is a strictly positive harmonic function for $\mathcal{L}_K$. With $K \neq \emptyset$, the killing term $c_K$ is a non-trivial function on every connected component of $X \setminus K$ by connecteness of $b$. Consequently, we have $0=\sum \mathcal{L}_K u_K < \sum c_K u_K$. Now, the statement follows from Theorem~\ref{thm:abstractmain}, applied to the connected components of the graph $(b_K,c_K)$, together with the harmonic function $u_K$. 
\end{proof}

\section{Expansion of $w_{\mu}$ for $d\geq 3$}\label{sec:proofdge3}
In this section we show an expansion for  
\[
w_{\mu}(x) = \frac{\Delta_{\mu} G_{\mu}^{1/2}(x)}{G_{\mu}^{1/2}(x)}.
\]
for  {\rw}s $\mu$ over $\Z^d$ with $d \geq 3$. For the simple random walk, a  corresponding expansion has already been obtained in \cite{KPP18}.

\medskip

We first need some preparation on the transformation of the distance function. Define the bilinear map 
\[
\langle \cdot,\, \cdot \rangle_{\mu}\colon\R^d \times \R^d \to \R, \qquad \langle x,\, y \rangle_{\mu} = (\det \Sigma_{\mu})^{\frac{1}{d}} \, \langle x,\, \Sigma_{\mu}^{-1}y \rangle.
\]
Note that $\langle \cdot,\, \cdot \rangle_{\mu}$ defines a scalar product on $\R^d$ since $\Sigma_{\mu}^{-1}$ is symmetric and positive definite. In the following, we write $$\sigma_{\mu} = (\det \Sigma_{\mu})^{\frac{1}{2d}}.$$
Recall that we have set $$x_{\mu} = \sigma_{\mu} \, \Sigma_{\mu}^{-1/2}x$$ before. We write $|\cdot|_{\mu}$ for the norm associated with $\langle \cdot, \cdot \rangle_{\mu}$ and observe that this definition is consistent with the notion introduced before Theorem~\ref{thm:MAIN}, since
\begin{align*} 
	\langle x,y \rangle_{\mu} &= \sigma_{\mu}^2 \, {\langle x,\, \Sigma_{\mu}^{-1}y \rangle} =   \, {\langle \sigma_{\mu} \Sigma_{\mu}^{-1/2} x,\, \sigma_{\mu}\Sigma_{\mu}^{-1/2}y \rangle} =  \langle x_{\mu}, y_{\mu} \rangle.  
\end{align*} 
Note further that by the equivalence of the norms ${|\cdot|}$ and ${|\cdot|}_{\mu}$, we have $\mathcal{O}\big(|x|_{\mu}^{\alpha} (\log|x|_{\mu})^{\beta}  \big) = \mathcal{O}\big( |x|^{\alpha} (\log|x|)^{\beta} \big)$ for all $\alpha,\beta \leq 0$. 
We will use this fact repeatedly below.

We will use the following special case of the fine asymptotics for  $G_{\mu}$ by Uchiyama.

\begin{theorem}[Special case of {\cite{Uch98}}] \label{thm:UchiyamaFukai2}
	Let $\mu$ be a {\rw} over $\Z^d$, $d \geq 3$. If $\|\mu\|_{d+3} <\infty$, then, as $|x| \to \infty$, 
	\[
	G_{\mu}(x) = \frac{c_{d,\mu}}{|x|_{\mu}^{d-2}} + \frac{\vartheta(x_{\mu})}{|x|_{\mu}^d} + \mathcal{O}(|x|^{-(d+2)}),
	\] 
	where $c_{d,\mu}$ is a positive constant and $\vartheta(x) = q(\omega_i^x)$, where $q:\R^d \to \R$ is a polynomial and $\omega_i^x$, $1 \leq i \leq d$, is the $i$-th coordinate of $x/|x|$.
\end{theorem}

\begin{remark}
	For the simple random walk, the polynomial $q$ is given by $q(z) = c_d^{(2)} \! \big( \sum_{j=1}^d z_j^4 - \frac{3}{d+2} \big)$ for some constant $c_d^{(2)} > 0$.
\end{remark}

\begin{proof}
	This follows immediately from \cite[Theorem~2]{Uch98} in combination with the comment on symmetric random walks made on p.~217. 
\end{proof}

We obtain the following expansion for $w_{\mu}$.

\begin{proposition} \label{prop:asymp_w_d>2}
	Suppose that $\mu$ is a {\rw} on $\mathbb{Z}^d$ for $d \geq 3$ such that $\|\mu\|_{d+3} < \infty$. Then, asymptotically as $|x| \to \infty$,
	\begin{align*}
		w_{\mu}(x) = \frac{(d-2)^2 \sigma_{\mu}^2}{8 |x|_{\mu}^2} + \mathcal{O}\big(|x|^{-4}\big).
	\end{align*}
\end{proposition}

\begin{remark}
	For the simple random walk, with the explicit formula of the higher-order term (see the Remark following Theorem~\ref{thm:UchiyamaFukai2}) it should be possible to derive the asymptotic expansion for $w = \frac{\Delta G^{1/2}}{G^{1/2}}$ with one more higher order term. After consultation with AI we conjecture it to be
	\begin{align*}
		w_{\mu}(x) = \frac{(d-2)^2}{8 d|x|^2} + \frac{Q(\omega_i^x)}{\abs{x}^4} + \mathcal{O}\big(|x|^{-6}\big),
	\end{align*}
	with $Q$ given as 
	\begin{align*}
		Q(z) = \frac{(d-2)^2}{16d} \left( \frac{5(d+2)^2}{8} \sum_{j=1}^d z_j^4 -(2d+3) \right).
	\end{align*}
	Asymptotically, this would make this discrete weight larger than its continuous counterpart on the coordinate axes, and smaller on the diagonals. 
\end{remark}

We need the following lemma which will also be useful in the next section when we consider the $d=2$ case. 

\begin{lemma}\label{lem:normcovariance}
	Let $\mu$ be a {\rw} over $\Z^d$ with $\|\mu\|_2 < \infty$. Then, for all $x \in \Z^d$, $x \neq 0$,
	\[
	\sum_{v \in \Z^d} \mu(v) \frac{\langle x_{\mu}, v_{\mu} \rangle^2}{|x_{\mu}|^2} = \sigma_\mu^2. 
	\]
\end{lemma}

\begin{proof}
	We will make use of the definition $\Sigma_{\mu} = \sum_{v \in \Z^d} \mu(v)vv^T$. Using this, we obtain for every  symmetric matrix $A \in \R^{d \times d}$ and $x \in \Z^d$,  
	\begin{align*}
		\sum_{v \in \Z^d} \mu(v) \langle x,Av \rangle^2 &= \sum_{v \in \Z^d} \mu(v) \big( x^T Av \big)^2 = \sum_{v \in \Z^d} \mu(v) \big( x^T Av \big)\big( v^T A x \big) \\
		&= x^T A \sum_{v \in \Z^d} \mu(v) vv^T Ax = x^T A \Sigma_{\mu} A x.
	\end{align*}   Consequently, with $A= \Sigma_{\mu}^{-1}$, $y_{\mu} = \sigma_{\mu} \, \Sigma_{\mu}^{-1/2}y$ and $\sigma_{\mu} = (\det \Sigma_{\mu})^{\frac{1}{2d}}$,  we get
	\[
		\sum_{v \in \Z^d} \mu(v) \frac{\langle x_{\mu}, v_{\mu} \rangle^2}{|x_{\mu}|^2}  = \sigma_{\mu}^2 \sum_{v \in \Z^d} \mu(v) \frac{\langle x, \Sigma_{\mu}^{-1}v \rangle^2}{\langle x, \Sigma_{\mu}^{-1} x \rangle} = \sigma_{\mu}^2 \frac{x^T \Sigma_{\mu}^{-1}x}{\langle x, \Sigma_{\mu}^{-1}x \rangle} = \sigma_{\mu}^2.\hfill \qedhere
	\]
\end{proof}

Let now  $\mu$ be a {\rw} on $\mathbb{Z}^d$, $d \geq 3$, such that $\|\mu\|_{d+3} < \infty$. Note that under scaling $w_{c\mu} = cw_\mu$, $\sigma^2_{c\mu} = c \sigma^2_{\mu}$ and $|\cdot|_{c\mu} = |\cdot|_{\mu}$ and so we can assume without loss of generality that $\mu$ is normalized, i.e., a random walk.
Then by Theorem~\ref{thm:UchiyamaFukai2} the $G_\mu$ admits the asymptotic expansion
\begin{align*}
	G_{\mu}(x) = \frac{c_{d,\mu}}{|x|_{\mu}^{d-2}} + \frac{\vartheta(x_{\mu})}{|x|_{\mu}^d} + \mathcal{O}(|x|^{-(d+2)}).
\end{align*}

\medskip

We consider the sets $A_x = \{v \in \mathbb{Z}^d \mid 2|v|_\mu< |x|_\mu \}$, $x \in \Z^d$, and for a function $F$ in $x$ and $v$ write $$F(x,v) = \mathcal{O}_A(|v|^m |x|^{-n})$$ if there exists a constant $C > 0$ with $\abs{F(x,v)} \leq C |v|^m |x|^{-n}$ for all $x \neq 0$ and $v \in A_x$.
Using this notation, we prove the following two lemmas.
\begin{lemma}
	\label{lem:psi}
	Let $\psi(x) = {\vartheta(x_\mu)}{|x|_\mu^{-2}}$ where $\vartheta$ is the function from the asymptotic expansion in Theorem~\ref{thm:UchiyamaFukai2}. Then, $\nabla \psi(x) v = \mathcal{O}(|v| |x|^{-3})$ and $v^T H_\psi(x) v = \mathcal{O}(|v|^2|x|^{-4})$, where $H_\psi$ is the Hessian matrix, and 
	\begin{align*}
		\psi(x+v) = \psi(x) + \nabla \psi(x) v + \mathcal{O}_A\left(|v|^2 |x|^{-4}\right).
	\end{align*}
\end{lemma}

\begin{proof}
	We claim that the Hessian matrix of $\psi$ satisfies
	\begin{align*}
		H_\psi(\lambda x) = \lambda^{-4} H_\psi(x)
	\end{align*}
	for all $\lambda > 0$ and $x \neq 0$. Denote by $m_\lambda:\R \rightarrow \R$, $t \mapsto \lambda t$ the multiplication operator. Note that $\psi(\lambda x) = \lambda^{-2} \psi(x)$, i.e.,\@ $\psi \circ m_\lambda = \lambda^{-2} \psi$, and so for all $i, j = 1, \ldots ,d$ it holds by the chain rule,
	\begin{align*}
		\lambda^{-2} \partial_i \partial_j \psi(x) = \partial_i \partial_j (\psi \circ m_\lambda)(x) = \lambda^2 \partial_i \partial_j \psi(\lambda x). 
	\end{align*}
	The claim about the Hessian matrix follows. Since $\psi$ is smooth, we can find a constant $C > 0$ such that $| v^T H_\psi(x) v | \leq C$ for all $x, v \in \Z^d$ with $|x| = |v| = 1$. From that it follows that 
	\begin{align*}
		\frac{|v^T H_\psi(x) v|}{|v|^2 |x|^4} = \abs{ \left( \frac{v}{|v|} \right)^T  H_\psi \!\left( \frac{x}{|x|} \right) \frac{v}{|v|} } \leq C
	\end{align*}
	for all $x, v \in \Z^d \backslash \set{0}$. This proves $v^T H_\psi(x) v = \mathcal{O}(|v|^2|x|^{-4})$. 
	
	The asymptotic order of $\nabla \psi(x) v$ is proven similarly. By  Taylor's formula there exists $u_{x,v}$ on the line segment connecting $x$ and $x + v$ such that
	\begin{align*}
		\psi(x+v) = \psi(x) + \nabla \psi(x) v + \frac{1}{2} v^T H_\psi(u_{x,v}) v.
	\end{align*}
	If $v \in A_x$, then $(1/2) |x|_\mu \leq |u_{x,v}|_\mu \leq (3/2) |x|_\mu$ and it follows that
	\[
		\psi(x+v) - \psi(x) - \nabla \psi(x) v = \mathcal{O}_A\left( |v|^2 |x|^{-4} \right).\hfill\qedhere
	\]
\end{proof}

\begin{lemma} \label{lem:power_expansions_G}
	Let $r(x,v) = \frac{G_\mu(x+v) - G_\mu(x)}{G_\mu(x)}$. Then,
	\begin{align*}
		r(x,v) &= O_1(x,v) + E_2(x,v) + O_3(x,v) + \mathcal{O}_A\big(|v|^4 |x|^{-4}\big), \\
		r(x,v)^2 &= O_1^2(x,v) + O_1(x,v)E_2(x,v) + \mathcal{O}_A\big(|v|^4 |x|^{-4}\big), \\
		r(x,v)^3 &= O_1^3(x,v) + \mathcal{O}_A\big(|v|^4 |x|^{-4}\big), \\
		r(x,v)^4 &= \mathcal{O}_A\big(|v|^4 |x|^{-4}\big),
	\end{align*}
	where $O_i$ are odd functions in $v$ of asymptotic order $\mathcal{O}(|v|^i|x|^{-i})$ and $E_2$ is an even function in $v$ of asymptotic order $\mathcal{O}(|v|^2 |x|^{-2})$. The function $O_1$ is explicitly given by $O_1(x,v) = (2-d) \frac{\langle x,v \rangle_\mu}{\abs{x}_\mu^2}$.
\end{lemma}

\begin{proof}
	Let $\psi(x) = {\vartheta(x_\mu)}{|x|_\mu^{-2}}$ be as in the previous lemma. Then, by Theorem \ref{thm:UchiyamaFukai2}, the function $G_\mu$ satisfies 
	\begin{align*}
		G_{\mu}(x) &= \frac{c_{d,\mu}}{|x|_{\mu}^{d-2}} + \frac{\vartheta(x_{\mu})}{|x|_{\mu}^d} + \mathcal{O}(|x|^{-(d+2)}) \\
		&= c_{d,\mu} |x|_\mu^{2-d} \left( 1 + c_{d,\mu}^{-1} \psi(x) +  \mathcal{O}(|x|^{-4}) \right),
	\end{align*}
	which implies, since $\psi(x)=\mathcal{O}(|x|^{-2})$, 
	\begin{align*}
		G_\mu^{-1}(x) = c_{d,\mu}^{-1} |x|_\mu^{d-2} \left( 1 - c_{d,\mu}^{-1} \psi(x) + \mathcal{O}(|x|^{-4}) \right).
	\end{align*}
	Further,
	\begin{align*}
		G_{\mu}(x+v) = c_{d,\mu} |x+v|_\mu^{2-d} \left( 1 + c_{d,\mu}^{-1} \psi(x+v) + \mathcal{O}(|x+v|^{-4}) \right).
	\end{align*}
	We take the product of those two previous expansions to get an expansion for $\frac{G_\mu(x+v) - G_\mu(x)}{G_\mu(x)}$ but we want it as an expansion in $\mathcal{O}_A$ instead. We can replace $\mathcal{O}(|x|^{-4})$ by $\mathcal{O}_A(|x|^{-4})$ and since $(1/2) |x|_\mu \leq |x+v|_\mu \leq (3/2) |x|_\mu$ if $v \in A_x$, we can also replace $\mathcal{O}(|x+v|^{-4})$ by $\mathcal{O}_A(|x|^{-4})$. Additionally, we may always replace $\mathcal{O}_A(|v|^m |x|^{-4})$ by $\mathcal{O}_A(|v|^4 |x|^{-4})$, if $m \leq 4$, since $|\cdot|_\mu$ is bounded away from zero on $\Z^d\setminus\{0\}$. Finally, we can always replace $\mathcal{O}_A(|v|^n |x|^{-n})$ by $\mathcal{O}_A(|v|^4 |x|^{-4})$, if $n \geq 4$ and $v \in A_x$.
	
	With that in mind and with the expansion $\psi(x+v) = \psi(x) + \nabla \psi(x) v + \mathcal{O}_A\left(|v|^2 |x|^{-4}\right)$ from the previous lemma at hand, we  derive the expansion	
	\begin{align*}
		G_{\mu}(x+v) = c_{d,\mu} |x+v|_\mu^{2-d} \left( 1 + c_{d,\mu}^{-1} \psi(x) + c_{d,\mu}^{-1} \nabla \psi(x) v + \mathcal{O}_A(|v|^4 |x|^{-4}) \right),
	\end{align*}
	and then it follows that 
	\begin{align*}
		\frac{G_\mu(x+v)}{G_\mu(x)} = \left( \frac{|x+v|_\mu^2}{|x|_\mu^2} \right)^{\frac{2-d}{2}} \left( 1 + c_{d,\mu}^{-1} \nabla \psi(x) v + \mathcal{O}_A(|v|^4 |x|^{-4}) \right),
	\end{align*}
	where we use that $\psi^2(x) = \mathcal{O}(|x|^{-4})$ and $\psi(x) \nabla \psi(x) v = \mathcal{O}_A(|v||x|^{-5})$ which is both in $\mathcal{O}_A(|v|^4 |x|^{-4})$.
	When we expand
	\begin{align*}
		\left( \frac{|x+v|_\mu^2}{|x|_\mu^2} \right)^{\frac{2-d}{2}} = \left( 1 + \frac{2\langle x,v \rangle_{\mu} + |v|_{\mu}^2}{|x|_{\mu}^2} \right)^{\frac{2-d}{2}}
	\end{align*}
	according to $(1 + r)^s = 1 + sr + \binom{s}{2}r^2 + \binom{s}{3} r^3 + \mathcal{O}(r^4)$, $s = \frac{2-d}{2}$, this yields a sum over products over the terms $\frac{2\langle x,v \rangle_{\mu}}{|x|_\mu^2}$ and $\frac{|v|_{\mu}^2}{|x|_{\mu}^2}$. The first term is odd in $v$ and of order $\mathcal{O}(|v| |x|^{-1})$, and the second one is even in $v$ and of order $\mathcal{O}(|v|^2 |x|^{-2})$. Their products are in $\mathcal{O}(|v|^n |x|^{-n})$ for some $n \in \N$, which on $A_x$ is in $\mathcal{O}(|v|^4 |x|^{-4})$ as soon as $n \geq 4$. We denote the sum over all $\mathcal{O}(|v|^2 |x|^{-2})$ products by $E_2(x,v)$, which is even in $v$, and the sum over all $\mathcal{O}(|v|^3 |x|^{-3})$ products as $\tilde{O}_3(x,v)$, which is odd in $v$. The only terms that do not belong to either, or to the general error term $\mathcal{O}_A(|v|^4 |x|^{-4})$, are the constant $1$ term and $(2-d) \frac{\langle x,v \rangle_\mu}{\abs{x}_\mu^2}$ which we denote by $O_1(x,v)$ and it is also odd in $v$. 
	With this notation we get the asymptotic expansion 
	\begin{align*}
		\left( \frac{|x+v|_\mu^2}{|x|_\mu^2} \right)^{\frac{2-d}{2}} =  1 + O_1(x,v) + E_2(x,v) + \tilde{O}_3(x,v) + \mathcal{O}_A(|v|^4 |x|^{-4})
	\end{align*}
	and it follows that
	\begin{multline*}
		\frac{G_\mu(x+v)}{G_\mu(x)} = \left( 1 + O_1(x,v) + E_2(x,v) + \tilde{O}_3(x,v) + \mathcal{O}_A(|v|^4 |x|^{-4}) \right) \\
		\hfill \cdot \left( 1 + c_{d,\mu}^{-1} \nabla \psi(x) v + \mathcal{O}_A(|v|^4 |x|^{-4}) \right) \\
		\hfill= 1 + O_1(x,v) + E_2(x,v) + \tilde{O}_3(x,v) + c_{d,\mu}^{-1} \nabla \psi(x) v + \mathcal{O}_A(|v|^4 |x|^{-4}).
	\end{multline*}
	Note that $\nabla \psi(x) v$ is odd in $v$ and $\nabla \psi(x) v = \mathcal{O}_A(|v|^3 |x|^{-3})$ by the previous lemma. If we set $O_3(x,v) = \tilde{O}_3(x,v) + c_{d,\mu}^{-1} \nabla \psi(x) v$ and subtract $1$, then we finish the proof of the expansion for $r(x,v)$. The expansions for the higher powers follow directly from the first one. 
\end{proof}

\begin{proof}[Proof of Proposition~\ref{prop:asymp_w_d>2}]
	First, we write $w_\mu$ as
	\begin{align*}
		w_\mu(x) = \sum_{v \in \Z^d} \mu(v) \left( 1 - \sqrt{ 1 + \frac{G_\mu(x+v) - G_\mu(x)}{G_\mu(x)} } \right).
	\end{align*}
	We use the bound $\abs{1-\sqrt{1+r}} \leq \abs{r}$ 
	valid for $r \geq -1$ to estimate
	\begin{align*}
		\abs{1 - \sqrt{ 1 + \frac{G_\mu(x+v) - G_\mu(x)}{G_\mu(x)} }} \leq \abs{\frac{G_\mu(x+v) - G_\mu(x)}{G_\mu(x)}} \leq C |x|^{d-2},
	\end{align*}
	where we  used boundedness of $G_\mu$ in the numerator and the asymptotic expansion of  $G_\mu$ from Theorem~\ref{thm:UchiyamaFukai2} in the denominator. 
	It follows that
	\begin{align*}
		\abs{x}^4\!
		\! \sum_{v \in X \backslash A_x}\! \mu(v) \left( 1 - \sqrt{ 1 + \frac{G_\mu(x+v) - G_\mu(x)}{G_\mu(x)} } \right) &\leq 2^{d+2} C \!\!\sum_{v \in X \backslash A_x}\! \mu(v) \abs{v}^{d+2}, 
	\end{align*}
	which under the moment condition $\|\mu\|_{d+3} < \infty$ implies that 
	\begin{align*}
		w_\mu(x) = \sum_{v \in A_x} \mu(v) \left( 1 - \sqrt{ 1 + \frac{G_\mu(x+v) - G_\mu(x)}{G_\mu(x)} } \right) + \mathcal{O} \left( \abs{x}^{-4} \right).
	\end{align*}
	Similarly, one also shows that
	\begin{align*}
		\sum_{v \in A_x} \mu(v) \frac{G_\mu(x+v) - G_\mu(x)}{G_\mu(x)} = -\frac{\Delta G_\mu(x)}{G_\mu(x)} + \mathcal{O}\left(|x|^{-4}\right) = \mathcal{O}\left(|x|^{-4}\right),
	\end{align*}
	where we used that $G_\mu$ is harmonic on $\Z^d \setminus \{0\}$, \cite[p.\@ 279]{Spitzer76}.
	
	Next, we use the asymptotic expansion for $r\to 0$
	\begin{align*}
		1 - \sqrt{1 + r} &= - \frac{1}{2} r + \frac{1}{8} r^2 - \frac{1}{16} r^3 + \mathcal{O}(r^4)
	\end{align*}
	and the estimates for $r(x,v) = \frac{G_\mu(x+v) - G_\mu(x)}{G_\mu(x)}$ from Lemma~\ref{lem:power_expansions_G} to obtain
	\begin{multline*}
		1 - \sqrt{ 1 + \frac{G_\mu(x+v) - G_\mu(x)}{G_\mu(x)} } \\\hspace{-2cm}= - \frac{1}{2} r(x,v) + \frac{1}{8} r(x,v)^2 - \frac{1}{16} r(x,v)^3 + \mathcal{O}_A(|v|^4 |x|^{-4}) \\
		= - \frac{1}{2} \frac{G_\mu(x+v) - G_\mu(x)}{G_\mu(x)} + \frac{1}{8} \left( O_1^2(x,v) + O_1(x,v)E_2(x,v) \right) \\
		\quad\quad - \frac{1}{16} O_1^3(x,v) + \mathcal{O}_A\left( |v|^4 |x|^{-4} \right).
	\end{multline*}
	As $O_1 E_2$, $O_1^3$ are odd in $v$ and $A_x=\{2|v|_\mu<|x|_\mu\}$ is symmetric (and finite),
	\begin{align*}
		\sum_{v \in A_x} \mu(v) O_1(x,v)E_2(x,v) = 0 = \sum_{v \in A_x} \mu(v) O_1^3(x,v).
	\end{align*}
	Furthermore, since $\|\mu\|_{4}\leq \|\mu\|_{d+3} < \infty$,
	\begin{align*}
		\sum_{v \in A_x} \mu(v) \mathcal{O}_A\left( |v|^4 |x|^{-4} \right) = \mathcal{O} \left( \abs{x}^{-4} \right).
	\end{align*}
	We have $O_1(x,v) = (2-d) \frac{\langle x,v \rangle_\mu}{\abs{x}_\mu^2}$ by Lemma~\ref{lem:power_expansions_G}, so,
	\begin{align*}
		\sum_{v \in A_x} \mu(v) O_1^2(x,v) &= (2-d)^2 |x|_\mu^{-2} \sum_{v \in A_x} \mu(v) \frac{\langle x,v \rangle_\mu^2}{\abs{x}_\mu^2} \\
		&= (2-d)^2 |x|_\mu^{-2} \left( \sum_{v \in \Z^d} \mu(v) \frac{\langle x,v \rangle_\mu^2}{\abs{x}_\mu^2} + \mathcal{O} \left( \abs{x}^{-2} \right) \right) \\
		&= (2-d)^2 \sigma_\mu^2 |x|_\mu^{-2} + \mathcal{O} \left( \abs{x}^{-4} \right),
	\end{align*}
	where we used the definition of $A_x$ and $\|\mu\|_{2}<\infty$  in the second line as well as 
	 Lemma~\ref{lem:normcovariance} in the last line. Putting everything together, this yields
	\begin{align*}
		w_\mu(x) &= \sum_{v \in A_x} \mu(v) \left( 1 - \sqrt{ 1 + \frac{G_\mu(x+v) - G_\mu(x)}{G_\mu(x)} } \right) + \mathcal{O} \left( \abs{x}^{-4} \right) \\
		&= \frac{1}{8} \sum_{v \in A_x} \mu(v) O_1^2(x,v) + \mathcal{O} \left( \abs{x}^{-4} \right) \\
		&= \frac{(2-d)^2 \sigma_\mu^2}{8 |x|_\mu^{2}} + \mathcal{O} \left( \abs{x}^{-4} \right).\hfill\qedhere
	\end{align*}
\end{proof}

\section{Expansion of $w_{\mu}$ for $d=2$}\label{sec:proofd=2}

In this section, we prove the expansion formula for the weight
\[
w_{\mu}\colon \Z^2 \setminus \{0\} \to [0, \infty), \qquad w(x) = \frac{\Delta_{\mu} \, a_{\mu}^{1/2}(x)}{a_{\mu}^{1/2}(x)},
\]	
provided that $\mu$ is  symmetric  and satisfies an appropriate moment condition.

Our proof relies strongly on an asymptotic expansion of the potential kernel due to results of Fukai and Uchiyama \cite{FU96,Uch98}, see also  \cite{Spitzer76}.

\begin{theorem}[Special case of \cite{FU96, Uch98}] \label{thm:UchiyamaFukai}
	Let $\mu$ be a  random walk over $\Z^2$. Then $a_{\mu}$ tends to infinity, i.e.,\@
	\[
	\lim_{|x| \to \infty} a_{\mu}(x) = \infty.
	\]
	If even  $\|\mu\|_6 < \infty$ and $\mu$ is symmetric,
	 then 
	the potential kernel $a_{\mu}$ has the asymptotic expansion 
	\[
	a_{\mu}(x) = \frac{1}{\pi \sigma_{\mu}^2} \log|x|_{\mu} \, + \, \kappa_{\mu} \, + \,  \frac{\vartheta(x_{\mu})}{|x|_{\mu}^2} + \mathcal{O}\Big( |x|^{-4} \Big)
	\]
	as $|x| \to \infty$, where $\kappa_{\mu}$ is a constant  and $\vartheta(x) = p(\omega_1^x, \omega_2^x)$, where $p:\R^2 \to \R$ is a polynomial and  $\omega^x_i$, $i \in \{1,2\}$, is the $i$-th coordinate of $x/|x|$.
	\end{theorem}

	\begin{proof}
		See \cite[Theorem~2]{FU96} and \cite[Theorem~1]{Uch98}, together with the comment on p.\@ 217 in \cite{Uch98} on vanishing of the $\mathcal{O}(|x|^{-1})$ and $\mathcal{O}(|x|^{-3})$ terms for {\rw}s.
	\end{proof}

	\begin{remark}[The constant $\kappa_{\mu}$]
		In certain situations of interest, one can make the terms in the expansion more explicit. For instance, for the simple random walk $\mu= 1/4 \cdot  1_{|x|=1}$, we have $$\kappa_{\mu} = \frac{\log 8 + 2\gamma}{\pi},$$ where $\gamma$ is the Euler constant and $$\vartheta(x) = (6\pi)^{-1} \big( 8(\omega_1^x \omega_2^x)^2 - 1 \big),$$ see \cite[Remark~2]{FU96}. Furthermore, the authors of \cite{FU96} point out that their proof is identical to the one in \cite[Section~12, P3]{Spitzer76} for a slightly less general case, where the  constants $\eta$ and $\kappa_{\mu}$ are shown to be positive.
	\end{remark}
	
\begin{proposition} \label{prop:asymp_w_d=2}
	Suppose that $\mu$ is a {\rw} on $\Z^2$ and assume that $\|\mu\|_{6} < \infty$. Then, asymptotically as $|x| \to \infty$, 
	\begin{align*}
		w_{\mu}(x) &=  \, \frac{\theta^2}{8\pi^2 \sqrt{\det \Sigma_{\mu}}}  \frac{1}{a_{\mu}(x)^2 |x|_{\mu}^2} \, + \, \mathcal{O} \Bigg( \frac{1}{|x|^{4} (\log |x|)^2} \Bigg), \\
		&= \frac{\sqrt{\det \Sigma_{\mu}}}{8|x|_{\mu}^2 (\log|x|_{\mu})^2} -  \frac{\kappa_{\mu} \pi \det \Sigma_{\mu}}{4 \theta |x|_{\mu}^2 (\log |x|_{\mu})^3}  + \mathcal{O}\Bigg( \frac{1}{|x|^2 (\log |x|)^4} \Bigg),
	\end{align*}
	where $\theta = \sum_{\Z^2} \mu$ and $\kappa_{\mu}$ is from the asymptotic expansion of $a_{\mu}$. 
\end{proposition}

We spend the rest of the section on proving this proposition. Let $\mu$ be a {\rw} on $\Z^2$ such that $\|\mu\|_{6}  < \infty$.  
As in the $d \ge 3$ case, we may assume that $\mu$ is normalized since $w_{c\mu} = cw_\mu$, $\Sigma_{c\mu} = c \Sigma_\mu$, $a_{c\mu} = a_\mu$, $\kappa_{c\mu} = \kappa_\mu$, and $|\cdot|_{c\mu} = |\cdot|_{\mu}$.

\begin{lemma}
	\label{lem:log_bound}
	There exists $C > 0$, $r > 1$ such that
	\begin{align*}
		\abs{ a_\mu(x+v) - a_\mu(x) } \leq C \log(\abs{v}_\mu)
	\end{align*}
	whenever $x \neq 0$ and $\abs{v}_\mu > r$.
\end{lemma}

\begin{proof}
	By the asymptomptic expansion of $a_\mu$, Theorem~\ref{thm:UchiyamaFukai},
	\begin{align*}
		a_{\mu}(x) &= \frac{1}{\pi \sigma_{\mu}^2} \log|x|_{\mu} \, + \, \kappa_{\mu} \, + \,  \frac{\vartheta(x_{\mu})}{|x|_{\mu}^2} + \mathcal{O}\Big( |x|^{-4} \Big) = \frac{1}{\pi \sigma_{\mu}^2} \log|x|_{\mu} \, + \mathcal{O}(1).
	\end{align*}
	There exists $\varepsilon > 0$ with $\abs{x}_\mu \geq \varepsilon$ for $x \neq 0$, and then, with $c = \frac{1}{\pi \sigma_{\mu}^2}$ and a suitable constant $M > 0$, it holds that
	\begin{align*}
		\abs{a_\mu(x+v) - a_\mu(x)} &\leq c \abs{ \log \left( \frac{\abs{x+v}_\mu}{\abs{x}_\mu} \right) } + M \leq c \log(1 + \varepsilon^{-1} \abs{v}_\mu) + M.
	\end{align*}
	Therefore,
	\begin{align*}
		\limsup_{\abs{v}_\mu \to \infty} \frac{\abs{a_\mu(x+v) - a_\mu(x)}}{\log(\abs{v}_\mu)} \leq c,
	\end{align*}
	and the claim follows. 
\end{proof}

As in the last section, we let $A_x = \{v \in \mathbb{Z}^2 \mid  2|v|_\mu< |x|_\mu \}$ for $x \in \Z^2$, and use the notation $F(x,v) = \mathcal{O}_A(|v|^m |x|^{-n})$ as before. Then, one can prove the following analogue of Lemma~\ref{lem:power_expansions_G}.

\begin{lemma} \label{lem:power_expansions}
	Let $r(x,v) = a_\mu(x+v) - a_\mu(x)$. Then, 
	\begin{align*}
		r(x,v) &= O_1(x,v) + E_2(x,v) + O_3(x,v) + \mathcal{O}_A\big(|v|^4 |x|^{-4}\big), \\
		r(x,v)^2 &= O_1^2(x,v) + O_1(x,v)E_2(x,v) + \mathcal{O}_A\big(|v|^4 |x|^{-4}\big), \\
		r(x,v)^3 &= O_1^3(x,v) + \mathcal{O}_A\big(|v|^4 |x|^{-4}\big), \\
		r(x,v)^4 &= \mathcal{O}_A\big(|v|^4 |x|^{-4}\big),
	\end{align*}
	where $O_i$ are odd functions in $v$ of asymptotic order $\mathcal{O}(|v|^i|x|^{-i})$ and $E_2$ is an even function in $v$ of asymptotic order $\mathcal{O}(|v|^2|x|^{-2})$. The function $O_1$ is explicitly given by $O_1(x,v) = \frac{1}{\pi \sigma_{\mu}^2}  \frac{\langle x,v \rangle_\mu}{\abs{x}_\mu^2}$.
\end{lemma}

\begin{proof}
	By Theorem~\ref{thm:UchiyamaFukai},
	\begin{align*}
		a_{\mu}(x) = \frac{1}{\pi \sigma_{\mu}^2} \log|x|_{\mu} \, + \, \kappa_{\mu} \, + \,  \frac{\vartheta(x_{\mu})}{|x|_{\mu}^2} + \mathcal{O}\Big( |x|^{-4} \Big).
	\end{align*}
	Let $\psi(x) = {\vartheta(x_{\mu})}{|x|_{\mu}^{-2}}$. By the same argument as in Lemma~\ref{lem:psi} it holds that
	\begin{align*}
		\psi(x+v) - \psi(x) = \nabla \psi(x) v + \mathcal{O}_A \left( \abs{v}^2 \abs{x}^{-4} \right).
	\end{align*}
	It follows that
	\begin{multline*}
		a_\mu(x+v) - a_\mu(x)= \frac{1}{\pi \sigma_{\mu}^2} \left( \log(\abs{x+v}_\mu) - \log(\abs{x}_\mu) \right) + \nabla \psi(x) v\\ + \mathcal{O}_A \left( \abs{v}^4 \abs{x}^{-4} \right), 
	\end{multline*}
	where we replaced $\mathcal{O}(|x+v|^{-4})$ and $\mathcal{O}_A \left( \abs{v}^2 \abs{x}^{-4} \right)$ by $\mathcal{O}_A \left( \abs{v}^4 \abs{x}^{-4} \right)$ which can be justified exactly as in the proof of Lemma~\ref{lem:power_expansions_G}. 
	Next, we can see that
	\begin{align*}
		\frac{1}{\pi \sigma_{\mu}^2} \left( \log(\abs{x+v}_\mu) - \log(\abs{x}_\mu) \right) = \frac{1}{2 \pi \sigma_{\mu}^2} \log \left( 1 + \frac{2\langle x,v \rangle_{\mu} + |v|_{\mu}^2}{|x|_{\mu}^2} \right), 
	\end{align*}
	and we expand this asymptotically according to $\log(1+r) = r - \frac{1}{2} r^2 + \frac{1}{3} r^3 + \mathcal{O}(r^4)$.
	Grouping the terms as in the proof of Lemma~\ref{lem:power_expansions_G} we can derive 
	\begin{align*}
		\frac{1}{\pi \sigma_{\mu}^2} \left( \log(\abs{x+v}_\mu) - \log(\abs{x}_\mu) \right) &= O_1(x,v) + E_2(x,v) \\
		&\quad\quad\quad + \tilde{O}_3(x,v) + \mathcal{O}_A(|v|^4 |x|^{-4}),
	\end{align*} 
	where $O_1$ is odd in $v$ and of order $\mathcal{O}(|v||x|^{-1})$, $E_2$ is even in $v$ and of order $\mathcal{O}(|v|^2|x|^{-2})$, and $\tilde{O}_3$ is odd in $v$ and of order $\mathcal{O}(|v|^3|x|^{-3})$ (the functions are not identical to those in Lemma~\ref{lem:power_expansions_G} though). 
	It follows that
	\begin{multline*}
		a_\mu(x+v) - a_\mu(x) \\= O_1(x,v) + E_2(x,v) + \tilde{O}_3(x,v) + \nabla \psi(x) v + \mathcal{O}_A \left( \abs{v}^4 \abs{x}^{-4} \right)
	\end{multline*}
	and setting $O_3(x,v) = \tilde{O}_3(x,v) + \nabla \psi(x) v$ finishes the proof of the expansion for $r$. The other expansions follow from the first one. 
\end{proof}

\begin{proof}[Proof of Proposition~\ref{prop:asymp_w_d=2}]
	The proof follows the same outline as the one of Proposition~\ref{prop:asymp_w_d>2}. 
	We write $w_\mu$ as
	\begin{align*}
		w_\mu(x) = \sum_{v \in \Z^d} \mu(v) \left( 1 - \sqrt{ 1 + \frac{a_\mu(x+v) - a_\mu(x)}{a_\mu(x)} } \right).
	\end{align*}
	We again use the bound  $\abs{1-\sqrt{1+r}} \leq \abs{r}$ and use Lemma~\ref{lem:log_bound} to obtain 
	\begin{align*}
		\abs{1 - \sqrt{ 1 + \frac{a_\mu(x+v) - a_\mu(x)}{a_\mu(x)} }} \leq \abs{\frac{a_\mu(x+v) - a_\mu(x)}{a_\mu(x)}} = \mathcal{O} \left(\frac{\log |v|}{\log |x|} \right)
	\end{align*}
	and
	\begin{multline*}
		\abs{x}^4 (\log |x|)^2 \sum_{v \in X \backslash A_x} \mu(v) \left( 1 - \sqrt{ 1 + \frac{a_\mu(x+v) - a_\mu(x)}{a_\mu(x)} } \right) \\
		\leq C \sum_{v \in X \backslash A_x} \mu(v) \abs{v}^{4} (\log |v|)^2.
	\end{multline*}
	Since $\|\mu\|_{6} < \infty$, this implies  
	\begin{align*}
		w_\mu(x) = \sum_{v \in A_x} \mu(v) \left( 1 - \sqrt{ 1 + \frac{a_\mu(x+v) - a_\mu(x)}{a_\mu(x)} } \right) + \mathcal{O} \left( \abs{x}^{-4} (\log |x|)^{-2} \right).
	\end{align*}
	and similarly, 
	\begin{align*}
		\sum_{v \in A_x} \mu(v) \frac{a_\mu(x+v) - a_\mu(x)}{a_\mu(x)}& = -\frac{\Delta a_\mu(x)}{a_\mu(x)} + \mathcal{O}\left(|x|^{-4} (\log |x|)^{-2} \right)\\
		& = \mathcal{O}\left(|x|^{-4} (\log |x|)^{-2}  \right),
	\end{align*}
	where we also used that $a_\mu$ is harmonic on $\Z^2 \setminus \{0\}$, see \cite{Spitzer76}. 
	Next, we use the asymptotic expansion 
	\begin{align*}
		1 - \sqrt{1 + t} &= - \frac{1}{2} t + \frac{1}{8} t^2 - \frac{1}{16} t^3 + \mathcal{O}(t^4)
	\end{align*}
	and Lemma~\ref{lem:power_expansions} to approximate using the notation $r(x,v) = a_\mu(x+v) - a_\mu(x)$
	\begin{align*}
		1 -& \sqrt{ 1 + \frac{a_\mu(x+v) - a_\mu(x)}{a_\mu(x)} } \\
		&= - \frac{1}{2} a_\mu(x)^{-1} r(x,v) + \frac{1}{8} a_\mu(x)^{-2} r(x,v)^2 \\
		&\quad\quad - \frac{1}{16} a_\mu(x)^{-3} r(x,v)^3 + \mathcal{O}_A(|v|^4 |x|^{-4} (\log |x|)^{-4}) \\
		&= - \frac{1}{2} \frac{a_\mu(x+v) - a_\mu(x)}{a_\mu(x)} + \frac{1}{8} a_\mu(x)^{-2} \left( O_1^2(x,v) + O_1(x,v)E_2(x,v) \right) \\
		&\quad\quad - \frac{1}{16}  a_\mu(x)^{-3} O_1^3(x,v) + \mathcal{O}_A\left( |v|^4 |x|^{-4} (\log |x|)^{-2} \right).
	\end{align*}
	Again, we have 
	\begin{align*}
		\sum_{v \in A_x} \mu(v) O_1(x,v)E_2(x,v) = \sum_{v \in A_x} \mu(v) O_1^3(x,v) = 0
	\end{align*}
	since $O_1 E_2$ and $O_1^3$ are odd in $v$ and $A_x$ is symmetric, and by the explicit form $O_1(x,v) = \frac{1}{\pi \sigma_{\mu}^2} \frac{\langle x,v \rangle_\mu}{\abs{x}_\mu^2}$ from Lemma~\ref{lem:power_expansions} and Lemma~\ref{lem:normcovariance} it follows by a short and straightforward computation that 
	\begin{align*}
		\sum_{v \in A_x} \mu(v) O_1^2(x,v) 
		&= \frac{1}{\pi^2 \sigma_{\mu}^4} |x|_\mu^{-2} \sum_{v \in \Z^d} \mu(v) \frac{\langle x,v \rangle_\mu^2}{\abs{x}_\mu^2} + \mathcal{O} \left( \abs{x}^{-4} (\log |x|)^{-2} \right) \\
		&= \frac{1}{\pi^2 \sigma_{\mu}^2} |x|_\mu^{-2} + \mathcal{O} \left( \abs{x}^{-4} (\log |x|)^{-2} \right).
	\end{align*}
	Since $\sigma_{\mu}^2 = \sqrt{\det \Sigma_{\mu}}$, this gives us the first expansion for $w_\mu$,
	\begin{align*}
		w_\mu(x) &= \sum_{v \in A_x} \mu(v) \left( 1 - \sqrt{ 1 + \frac{a_\mu(x+v) - a_\mu(x)}{a_\mu(x)} } \right) + \mathcal{O} \left( \abs{x}^{-4} (\log |x|)^{-2} \right) \\
		&= \frac{1}{8} \frac{1}{a_{\mu}(x)^2} \sum_{v \in A_x} \mu(v) O_1^2(x,v) + \mathcal{O} \left( \abs{x}^{-4} (\log |x|)^{-2} \right) \\
		&= \frac{1}{8 \pi^2 \sqrt{\det \Sigma_{\mu}}} \frac{1}{a_{\mu}(x)^2 |x|_{\mu}^2} + \mathcal{O} \left( \abs{x}^{-4} (\log |x|)^{-2} \right).
	\end{align*}
	To prove the second expansion, we set $\rho_\mu = \pi \sigma_{\mu}^2$ and remember that the asymptotic behavior of the potential kernel is 
	\begin{align*}
		a_{\mu}(x) &= \rho_\mu^{-1} \log|x|_{\mu} + \kappa_{\mu} + \mathcal{O}\big( |x|^{-2} \big) \\
		&= \frac{\log|x|_{\mu}}{\rho_\mu} \left( 1 + \frac{\rho_\mu \kappa_\mu}{\log |x|_\mu} + \mathcal{O} \left((\log |x|)^{-1} |x|^{-2} \right) \right).
	\end{align*}
	Applying the Taylor expansion $(1+y)^{-2} = 1 - 2y + \mathcal{O}(y^2)$, we  write $a_\mu^{-2}$ as
	\begin{align*}
		\frac{1}{a_{\mu}(x)^2} &= \frac{\rho_\mu^2}{(\log|x|_{\mu})^2} \left( 1 - \frac{2\rho_\mu \kappa_{\mu}}{\log|x|_{\mu}} + \mathcal{O}\Bigg(\frac{1}{(\log |x|)^{2}}\Bigg) \right) \\
		&= \frac{\rho_\mu^2}{(\log|x|_{\mu})^2} - \frac{2\rho_\mu^3 \kappa_{\mu}}{(\log|x|_{\mu})^3} + \mathcal{O}\Bigg(\frac{1}{(\log |x|)^{4}}\Bigg).
	\end{align*}
	Plugging this into the first asymptotic expansion yields
	\begin{align*}
		w_{\mu}(x) &= \frac{1}{8\pi\rho_\mu |x|_\mu^2} \left( \frac{\rho_\mu^2}{(\log|x|_{\mu})^2} - \frac{2\rho_\mu^3 \kappa_{\mu}}{(\log|x|_{\mu})^3} \right) + \mathcal{O}\Bigg( \frac{1}{|x|^2 (\log |x|)^4} \Bigg) \\
		&= \frac{\rho_\mu}{8\pi|x|_{\mu}^2 (\log|x|_{\mu})^2} - \frac{\rho_\mu^2 \kappa_{\mu}}{4\pi |x|_{\mu}^2 (\log |x|_{\mu})^3} + \mathcal{O}\Bigg( \frac{1}{|x|^2 (\log |x|)^4} \Bigg).
	\end{align*}
	Since $\frac{\rho_\mu}{\pi} = \sqrt{\det \Sigma_{\mu}}$, the claimed expansion holds. Note that in case $\mu$ is not normalized,  $\theta$ appears in the denominator of the second term in the expansion since both $w_{\mu}$ and $\rho_{\mu}$ scale linearly with respect to $\mu$.
\end{proof}

\section{Proofs of Theorem~\ref{thm:MAIN}, Corollary~\ref{cor:MAIN} and Theorem~\ref{thm:MAIN2}} \label{sec:proofs}

We will now prove the remaining main results stated in the introduction of this paper.

\medskip

The following two lemmas show that the technical conditions on $u=a_{\mu}$ and $u=G_{\mu}$ are satisfied in order to apply the general criteria provided in Theorem~\ref{thm:CRITERION} and Theorem~\ref{thm:abstractmain}.

\begin{lemma} \label{lem:propera}
	Let $\mu$ be a {\rw} over $\Z^2$  and let $a_{\mu}$ be the potential kernel. Then, $a_{\mu}$ is proper and satisfies a weak bounded oscillation condition.
\end{lemma}	

\begin{proof}
	We can clearly assume that $\mu$ is normalized. 
	It follows from the first part of Theorem~\ref{thm:UchiyamaFukai} that $\lim_{|x| \to \infty}a_{\mu}(x) = \infty$, so $a_{\mu} \in I_{\infty}(X)$. Further, $a_{\mu}$ belongs to the formal domain $\mathcal{F}$ of $\mathcal{L} = \Delta_{\mu}$ by \cite[Section~13, P3~(a)]{Spitzer76} in combination with \cite[Section~8, T1~(b)]{Spitzer76}.  
	 That $a_{\mu}$ is proper and satisfies a weak bounded oscillation condition now follows from Proposition~\ref{prop:criterionC0}.
\end{proof}

\begin{lemma} \label{lem:properG}
	Let $\mu$ be a {\rw} over $\Z^d$, $d \geq 3$. Then, $G_{\mu}$ is proper  
	and satisfies a weak bounded oscillation condition.
\end{lemma}

\begin{proof}
	Since $G_{\mu} \in C_0(X)$ by \cite[Section~24, P5]{Spitzer76}, this immediately follows from Proposition~\ref{prop:criterionC0}. 
\end{proof}

We are now in position to prove Theorem~\ref{thm:MAIN}.

\begin{proof}[Proof of Theorem~\ref{thm:MAIN}]
	Note that $a_{\mu}$ satisfies all the conditions in order to apply Theorem~\ref{thm:CRITERION} with $X=\Z^2$, $K= \{0\}$, $u=a_{\mu}$. Indeed, it is well-known that $a_{\mu}(x) > 0$ for $x \neq 0$ and $a_{\mu}(0)=0$, see e.g.\@ \cite[Section~11, P7]{Spitzer76}.
	Further, 
	 $a_{\mu}$ is {\em harmonic on $\Z^2 \setminus \{0\}$}, i.e.,\@ $\Delta_{\mu} a_{\mu}(x) = 0 $ for $x \neq 0$, cf.\@ \cite[Section~13, P3~(a)]{Spitzer76} in combination with \cite[Section~8, T1~(b)]{Spitzer76}. Together with Lemma~\ref{lem:propera}, we can apply Theorem~\ref{thm:CRITERION} with $u=a_{\mu}$, which shows that $w_u = w_{\mu}$ is a  null-critical Hardy weight  on $\Z^2 \setminus \{0\}$. The claimed asymptotic expansion in case $\|\mu\|_6 < \infty$ follows from Proposition~\ref{prop:asymp_w_d=2}. 
\end{proof}

We now prove Corollary~\ref{cor:MAIN}. 

\begin{proof}[Proof of Corollary~\ref{cor:MAIN}]
	We observe that $\theta =4$, $\Sigma_{\mu} = 2\mathrm{Id}_{\R^2}$ and $|x|=|x|_{\mu}$. Further, by the remark after Theorem~\ref{thm:UchiyamaFukai}, we have $\kappa_{\mu} = (\log 8 + 2\gamma)/\pi$, with $\gamma$ being the Euler constant. Thus, the corollary follows from Theorem~\ref{thm:MAIN}.  
\end{proof}
We give the proof of Theorem~\ref{thm:MAIN2} using  Theorem~\ref{thm:abstractmain} of this paper. 

	\begin{proof}[Proof I of Theorem~\ref{thm:MAIN2}]
		We apply Theorem~\ref{thm:abstractmain} with $X=\Z^d$, $K=\emptyset$ and $u=G_{\mu}$. Indeed, $G_{\mu}$ is strictly positive, superharmonic and harmonic on $\Z^d \setminus \{0\}$ \cite[p.\@ 279]{Spitzer76}, and so $\Delta_{\mu} G_{\mu} \in \ell^1(\Z^d)$. The remaining requirements are fulfilled by Lemma~\ref{lem:properG}. This gives null-criticality of $w_{\mu}$.
		Under the assumed moment condition, the expansion follows from Proposition~\ref{prop:asymp_w_d>2}.
	\end{proof}
	
Indeed, the null-criticality part of Theorem~\ref{thm:MAIN2} can be also achieved by the methods of \cite{HKPIII}. We show this using the terminology of \cite{HKPIII}.	\begin{proof}[Proof II of Theorem~\ref{thm:MAIN2}]
		As as a multiple of the Green function, $G_{\mu}$ is clearly a potential (i.e., in the image of the Green operator) with charge in $C_c(\Z^d)\subseteq \ell^1(\Z^d)$. 
				Hence, the assumptions of \cite[Theorem~1]{HKPIII} are satisfied with $X=\Z^d$, $m=1$, $(b,c) = (b_{\mu},0)$ and $u=G_{\mu}$, which gives null-criticality of $w_{\mu}$. The expansion follows as above from Proposition~\ref{prop:asymp_w_d>2}.
	\end{proof}

\medskip

\textbf{Acknowledgement.} 
The authors gratefully acknowledge the financial support by the Deutsche Forschungsgemeinschaft (DFG). 
Moreover, M.K.\@ and F.P.\@ thank the Israel Institute for Advanced Studies Jerusalem (IIAS), for its hospitality during their stay as fellows of the 2025/2026 research group ``Analysis, Geometry, and Spectral Theory of Graphs''.


\printbibliography
\end{document}